\documentclass[12pt]{amsart}

\newtheorem{theorem}{Theorem}[section]

\newtheorem{proposition}[theorem]{Proposition}

\newtheorem{definition}[theorem]{Definition}

\newcommand{\ncom}{\newcommand}
\ncom{\ep}{\epsilon}
\ncom{\rar}{\rightarrow}
\ncom{\thrar}{\twoheadrightarrow}
\ncom{\lrar}{\longrightarrow}
\ncom{\ov}{\overline}
\ncom{\what}{\widehat}

\newcommand{\ignore}[1]{}
\ncom{\m}{\mbox}
\ncom{\sta}{\stackrel}
\ncom{\C}{{\mathbb C}}
\ncom{\A}{{\mathbb A}}
\ncom{\Z}{{\mathbb Z}}
\ncom{\Q}{{\mathbb Q}}
\ncom{\R}{{\mathbb R}}
\ncom{\G}{{\mathbb G}}
\ncom{\HH}{{\mathbb H}}
\ncom{\al}{\alpha}
\ncom{\p}{{\mathbb P}}
\ncom{\N}{{\mathbb N}}
\ncom{\K}{{\mathbb K}}
\ncom{\X}{{\mathbb X}}
\ncom{\f}{\frac}
\ncom{\cA}{{\mathcal A}}
\ncom{\cB}{{\mathcal B}}
\ncom{\cD}{{\mathcal D}}
\ncom{\cDB}{{\mathcal D \mathcal B}}
\ncom{\cX}{{\mathcal X}}
\ncom{\cO}{{\mathcal O}}
\ncom{\cW}{{\mathcal W}}
\ncom{\cL}{{\mathcal L}}
\ncom{\cP}{{\mathcal P}}
\ncom{\cH}{{\mathcal H}}
\ncom{\cS}{{\mathcal S}}
\ncom{\cM}{{\mathcal M}}
\ncom{\cC}{{\mathcal C}}
\ncom{\cK}{{\mathcal K}}
\ncom{\cT}{{\mathcal T}}
\ncom{\cF}{{\mathcal F}}
\ncom{\cN}{{\mathcal N}}
\ncom{\cJ}{{\mathcal J}}
\ncom{\cV}{{\mathcal V}}
\ncom{\cZ}{{\mathcal Z}}
\ncom{\cU}{{\mathcal U}}
\ncom{\cSU}{{\mathcal S \mathcal U}}
\ncom{\cG}{{\mathcal G}}
\ncom{\cQ}{{\mathcal Q}}
\ncom{\cR}{{\mathcal R}}
\ncom{\cY}{{\mathcal Y}}
\ncom{\cE}{{\mathcal E}}
\ncom{\cI}{{\mathcal I}}
\ncom{\mylabel}[1]{{\rm (#1)}\label{#1}}
\ncom{\Hom}{{\textit{Hom}}}
\ncom{\eop}{{\hfill $\Box$}}
\begin{document}
\baselineskip=16pt

\title[Tertiary classes on a smooth manifold]{Tertiary classes for a one-parameter variation of flat connections on a smooth manifold}
\author{Jaya NN Iyer}

\address{The Institute of Mathematical Sciences, CIT
Campus, Taramani, Chennai 600113, India}
\email{jniyer@imsc.res.in}

\footnotetext{Mathematics Classification Number: 53C55, 53C07, 53C29, 53.50. }

\begin{abstract} In this note, we extend the theory of Chern-Cheeger-Simons to construct canonical invariants for a one-parameter family of flat connections on a smooth manifold. These invariants lie in degrees $(2p-2)$-cohomology with $\C/\Z$-cohomology, for $p\geq 2$, and are shown to be rigid in a variation of paths (parametrising flat connections), in degrees at least  three.
\end{abstract}

\maketitle

\setcounter{tocdepth}{1}
\tableofcontents

%%%%%%%%%%%%%%%%%%%%%%%%%%%%%%%%%%%%%%%%%%%%%%%%%%%%%%%%%%%%%%%%%%%%%%%%%%%%%%%%
\section{Introduction}
%%%%%%%%%%%%%%%%%%%%%%%%%%%%%%%%%%%%%%%%%%%%%%%%%%%%%%%%%%%%%%%%%%%%%%%%%%%%%%%

Suppose $X$ is a smooth manifold and $E$ is a smooth complex vector bundle on $X$ of rank $n$. 
Consider a smooth connection $\nabla$ on $E$. The Chern-Weil theory defines the Chern forms $P_p(\nabla^2,...,\nabla^2)$, of degree $2p$ and for $p\geq 1$. Here $P_p$ denotes the $GL_n$-invariant polynomial 
of degree $p$ and $\nabla^2$ denotes the curvature form 
of the connection form $\nabla$. The Chern forms are closed and correspond to the de Rham Chern classes $c_p(E)\in H^{2p}_{dR}(X,\C)$. These are the primary invariants of $E$, and are independent
of the connection. In particular when $\nabla$ is flat, i.e., when 
$\nabla^2=0$, the Chern-Cheeger-Simons 
theory \cite{Chn-Sm}, \cite{Ch-Sm} defines cohomology classes 
$$
\hat{c_p}(E,\nabla) \in H^{2p-1}(X,\C/\Z),
$$
for $p\geq 1$.
These are the secondary invariants of $(E,\nabla)$. These invariants are known to be rigid in a family of flat connections \cite[Proposition 2.9, p.61]{Ch-Sm}. In other words, the cohomology 
class remains the same in a variation of flat connections.

In this note, we give a construction of \textit{tertiary classes}, associated to a variation of a flat connection $(E,\nabla)$ on a smooth manifold $X$. More precisely, 
we show:

\begin{theorem}
 Suppose $X$ is a smooth manifold and $E$ is a complex vector bundle as above. Let $\gamma:=\{\nabla_t\}_{t\in I}$ be a one parameter family of flat connections on $E$.
We obtain canonical invariants
$$
CS_p(E,\gamma) \,\in\, H^{2p-2}(X,\C/\Z),
$$
whenever $p\geq 2$.
\end{theorem}
In fact, we obtain differential characters 
$$
\hat{c_p}(E, \gamma) \, \in \,\widehat{H^{2p-1}}(X)
$$
associated to a one-parameter family $\gamma$ of smooth connections. Here $\widehat{H^{2p-1}}(X)$ denotes the differential cohomology introduced by Cheeger-Simons in \cite{Ch-Sm}. When the path $\gamma$ parametrises flat connections, the differential character lies in $H^{2p-2}(X,\C/\Z)$, and we denote it by 
$CS_p(E,\gamma)$.

We call the above new invariants associated to a one parameter family $\gamma$ of flat connections,  the \textit{tertiary classes} of $(E,\gamma)$.

It is natural to ask about the behaviour of these classes in a variation of paths. In other words, we would like to have a \textit{variational  formula} which gives the rigidity of the tertiary invariants 
in certain degrees.

We obtain the following result on the rigidity of the invariants.  

\begin{proposition}
 Suppose $\{\gamma_s\}_{s\in I}$ is a one parameter family of smooth paths of connections $\{\nabla_{s,t}\}$, joining  two flat connections $\nabla_0$ and $\nabla_1$ on a complex vector bundle $E$, on a smooth manifold $X$.
Consider the corresponding one parameter family of differential characters $\hat{c_p}(E, \gamma_s)$ in $\widehat{H^{2p-1}}(X)$.
Then we have the variational formula:
$$
\frac{d}{ds} \hat{c_p}(E, \gamma_s)\,=\, \int_I P_p(\f{d}{ds}\f{d}{dt} \nabla_{t,s} \wedge (\nabla_{s,t}^2)^{p-1} + (p-1)\f{d}{dt} \nabla_{s,t}\wedge\f{d}{ds}(\nabla_{s,t}^2) \wedge (\nabla_{s,t}^2)^{p-2}). 
$$

In particular, 
 if $\{\gamma_s\}$ is a one parameter  family of paths parametrising flat connections, then  the classes:
$$
CS_p(E,\gamma_s) \in H^{2p-2}(X,\C/\Z)
$$
are constant, whenever $p\geq 3$. In other words the tertiary invariants in degrees at least three, depend only on the homotopy class of a path $\gamma$. 

\end{proposition}

The proof and construction involve obtaining a canonical differential character which lifts the $(2p-1)$-form associated to a one parameter variation $\{\nabla_t\}_{t\in I}$ of 
connections. This differential 
form is given as:
\begin{equation}\label{etainv}
\eta_p:= p.\int_I P_p(\f{d}{dt}\nabla_t, \nabla_t^2,...,\nabla_t^2). dt.
\end{equation}

 The difference of 
the Chern-Simons differential character $\hat{c_p}(E,\nabla_1)-\hat{c_p}(E,\nabla_0)$ is just the linear functional defined by the form $\eta_p$ (see \cite[Proposition 2.9, p.61]{Ch-Sm}). This form also gives the rigidity of the Chern-Simons classes for flat connections, in degrees at least two.

Hence it is natural to ask for a canonical lifting of the form $\eta_p$ as a differential character. 
The proof of this construction is by considering another path of connections which is a convex combination of the end points of the given path of flat connections, and then
use integration along fibres.

Further applications, including properties and other constructions of these invariants will be carried out in a future work. 
 In \S \ref{compactsupp}, we observe that the original Chern Simons classes are actually defined in the compactly supported cohomology.

{\Small Acknowledgement: This work is motivated by a question of M. Kontsevich on the rigidity property of the original Chern-Simons classes, arising from the above form $\eta_p$. We thank him for this question. We also thank M. Karoubi for pointing out to his works relevant to our constructions. This work was done at IHES, Paris, during our stay in 2013 and we are grateful to the institute for their hospitality and support.}

%%%%%%%%%%%%%%%%%%%%%%%%%%%%%%%%%%%%%%%%%%%%%%%%%%%%%%%%%%%%%%%%%%%%%%%%%%%%%%%%%%%%%%%%%%%%%%%%%%%%%%%%%%%%%%%%
\section{Preliminaries: original Chern-Cheeger-Simons theory}
%%%%%%%%%%%%%%%%%%%%%%%%%%%%%%%%%%%%%%%%%%%%%%%%%%%%%%%%%%%%%%%%%%%%%%%%%%%%%%%%%%%%%%%%%%%%%%%%%%%%%%%%%%%%%%%

Suppose $X$ is a smooth manifold, and $E$ is a smooth complex vector bundle on $X$. Our aim is to construct a canonical lifting of the differential form \eqref{etainv}, 
in the differential 
cohomology, defined by Cheeger and Simons in \cite{Ch-Sm}. We start by recalling the original constructions.

\subsection{Preliminaries}\cite{Ch-Sm}:
Suppose $X$ is a smooth manifold and $(E,\nabla)$ is a flat connection of rank $n$ on $X$. Recall the differential cohomology:
$$
\widehat{H^{p}}(X):= \{ (f,\alpha): f: Z_{p-1}(X)\rar \C/\Z,\, \delta(f)= \alpha, \alpha \m{ is a closed form and integral valued} \}.
$$
Here $Z_{p-1}(X)$ is the group of $(p-1)$-dimensional cycles, and $\delta$ is the coboundary map on cochains. The linear functional $\delta(f)$ is defined by integrating the form 
$\alpha$ against $p$-cycles, and it takes integral values on integral cycles.

This cohomology fits in an exact sequence:
\begin{equation}\label{diffcoh}
0\rar H^{p-1}(X,\C/\Z) \rar \widehat{H^{p}}(X) \rar A^{p}_\Z(X)_{cl} \rar 0.
\end{equation}

Here $A^{p}_\Z(X)_{cl}$ denotes the group of closed complex valued $p$-forms with integral periods.

Given a smooth connection $(E,\nabla)$ on $X$. The Chern forms $c_p(E,\nabla)$ are $2p$-differential forms on $X$.
Let $P_p$ denote the $GL_n$-invariant polynomial defining the Chern form, i.e., $c_p(E,\nabla):=P_p(\nabla^2,...,\nabla^2)$. Here $\nabla^2$ denotes the curvature form, and we say that $\nabla$ is \textit{flat} if $\nabla^2$ is identically zero.
 
We recall the following construction, to motivate the construction of tertiary classes.
\begin{theorem}\cite{Ch-Sm}\label{canDC}
The Cheeger-Simons construction defines a differential character, for any smooth connection $(E,\nabla)$:
$$
\hat{c_p}(E,\nabla)\, \in\, \widehat{H^{2p}}(X), \,\,\,p\geq 0.
$$
Moreover, if $\nabla$ is flat then 
$$
\hat{c_p}(E,\nabla)\, \in\, H^{2p-1}(X,\C/\Z).
$$
This is the same as the Chern-Simons class $CS_{p}(E,\nabla)$.
\end{theorem}
\begin{proof}
Let $P_p$ denote the degree $p$, $GL_n$-invariant polynomial defined on the set of matrices of size $n\times n$. 
The curvature form of $\nabla$ is $\Theta:=\nabla^2$. The Chern form in degree $2p$ is given as $P_p(\Theta,...,\Theta)$. The Chern-Weil theory says that the Chern form is closed and that it has integral periods. Hence it gives an element in the group $A^{2p}_\Z(X)_{cl}$.
Now use the universal smooth connection \cite{Narasimhan}, which lies on a Grassmannian $G(r,N)$, for large $N$. The differential cohomology of $G(r,N)$ is just the group $A^{2p}_\Z(X)_{cl}$, since the 
odd degree cohomologies of $G(r,N)$ are zero.  Hence the differential character $\hat{c_p}$ is the Chern form of the universal connection. The pullback of this element via the classifying map of $(E,\nabla)$ defines the differential character
$\hat{c_p}(E,\nabla)\in \widehat{H^{2p}}(X)$. 

Suppose $\nabla$ is flat, i.e., $\Theta=0$. Then the Chern form $P_p(\Theta,...,\Theta)=0$. Hence the differential character lies in the group
$H^{2p-1}(X,\C/\Z)$.

\end{proof}

%%%%%%%%%%%%%%%%%%%%%%%%%%%%%%%%%%%%%%%%%%%%%%%%%%%%%%%%%%%%%%%%%%%%%%%%%%%%%%%%%%%%%%%%%%%%%%%%%
\section{Tertiary classes for a one parameter family of flat connections}\label{analytichigherCS}

%%%%%%%%%%%%%%%%%%%%%%%%%%%%%%%%%%%%%%%%%%%%%%%%%%%%%%%%%%%%%%%%%%%%%%%%%%%%%%%%%%%

\subsection{Rigidity of Chern-Cheeger-Simons classes and the eta-form}

As in the previous section, suppose $(E,\nabla)$ is a smooth connection on a smooth manifold. Suppose
$\gamma:=\{\nabla_t\}_{t\in I}$ is a one parameter family of smooth connections on $E$ and $\nabla=\nabla_0$. Here $I:=[0,1]$ is the closed unit interval.

The family $\{\nabla_t\}_{t\in I}$ gives us a family of differential characters:
$$
\hat{c_p}(E,\nabla_t) \,\in\, \widehat{H^{2p}}(X)
$$
for $p\geq 1$ and $t\in [0,1]$.

The difference of the differential characters when $t=0$ and when $t=1$, is given by the following variational formula.

\begin{proposition}
With notations as above, we have the following equality:
$$
\hat{c_p}(E,\nabla_1)-\hat{c_p}(E,\nabla_0) \,=\, p. \int_{0}^1 P_p(\f{d}{dt}\nabla_t \wedge\nabla_t^{2(p-1)})_{| Z_{2p-1}(X)}
$$
for $p\geq 1$.
\end{proposition}
\begin{proof}
See \cite[p.61, Proposition 2.9]{Ch-Sm}.
\end{proof}

In particular, if $\{\nabla_t\}_{t\in I}$ is a family of flat connections then the degree $(2p-1)$-form 
\begin{equation}\label{etapform}
\eta_p:=  p. \int_{0}^1 P_p(\f{d}{dt}\nabla_t \wedge\nabla_t^{2(p-1)})
\end{equation}
is identically zero, when $p\geq 2$. This implies that
$$
\hat{c_p}(E,\nabla_1)\,=\,\hat{c_p}(E,\nabla_0).
$$
This gives the \textit{rigidity} of Chern-Simons classes
\begin{equation}\label{rigidCS}
CS_p(E,\nabla_1)\,=\,CS_p(E,\nabla_0)\,\in\, H^{2p-1}(X,\C/\Z)
\end{equation}
whenever $p\geq 2$.

%%%%%%%%%%%%%%%%%%%%%%%%%%%%%%%%%%%%%%%%%%%%%%%%%%%%%%%%%%%%%%%%%%%%%%%%%%%%%%%
\subsection{Canonical invariants lifting the difference of Chern-Simons classes}

Recall the coefficient sequence:
$$
0\rar \Z \rar \C \rar \C/\Z \rar 0.
$$
The associated long exact cohomology sequence is given by:
$$
\rar H^k(X,\C/\Z) \rar H^{k+1}(X,\Z) \rar H^{k+1}(X,\C) \rar H^{k+1}(X, \C/\Z) \rar.
$$
When $k+1=2p$, and $(E,\nabla)$ is a flat connection then the Chern-Simons class $CS_p(E,\nabla)$ is a canonical lifting of the (torsion) integral Chern class $c_p(E)$. Here we note that the de Rham Chern class
$c_p(E,\nabla)\in H^{2p}(X,\C)$ is identically zero, for $p\geq 1$.
We noticed in  proof of Theorem \ref{canDC}, that the differential cohomology is introduced as an intermediate object in the above cohomology sequence where the canonical invariants lie and they lift the de Rham Chern form.

When $k+1=2p-1$, and suppose we are given a path of flat connection $\gamma:=\{\nabla_t\}_{t\in I}$.  Moreover, M. Karoubi \cite[\S 3]{Karoubi} has defined the Chern-Simons classes in $H^{2p-1}(X,\C)$.
The difference of Chern-Simons classes, using the rigidity formula in \eqref{rigidCS}, gives us
$$
CS_p(E,\nabla_1)\,-\,CS_p(E,\nabla_0)\,=\,0\,\in\, H^{2p-1}(X,\C/\Z).
$$
This also holds in $H^{2p-1}(X,\C)$.
This difference is characterised by the eta-form $\eta_p$, in \eqref{etapform}.
Hence, we would like to construct a canonical lifting in $H^{2p-2}(X,\C/\Z)$ whose image in $H^{2p-1}(X,\C)$ is the difference of the Chern-Simons class, via differential characters. The differential character should be a canonical lifting of the eta form $\eta_p$, upto an exact form.

We make the constructions precise in the following subsection.

%%%%%%%%%%%%%%%%%%%%%%%%%%%%%%%%%%%%%%%%%%%%%%%%%%%%%%%%%%%%%%%%%%%%%%%%%%%
\subsection{Tertiary classes}

We now show that the above Chern-Cheeger-Simons construction together with rigidity of $CS_p(E,\nabla)$, give rise to classes in $H^{2p-2}(X,\C/\Z)$, corresponding to a  one parameter family of flat 
connections, whenever $p\geq 2$.

We show

\begin{theorem}\label{higherCS}
Suppose $(E,\nabla_i)$, for $i=0,1$, are two distinct flat connections on a smooth manifold $X$. 
Let $\gamma:=\{\nabla_t\}_{t\in I}$ be a one parameter family of flat connections with end-points  $\nabla_0$ and $\nabla_1$. Then
there are well-defined classes
$$
CS_p(E,\gamma)\,\in\, H^{2p-2}(X,\C/\Z),
$$
for each $p\geq 2$, 
\end{theorem}
\begin{proof}
The proof uses integration along the fibre for the smooth Deligne cohomology.

Let $\nabla_0$ and $\nabla_1$ denote two distinct flat connections on $E$. Consider the convex combination  of the two flat connections on $E$, as follows. 

Consider the projection $q:I\times X\rar X$, where $I=[0,1]$. Denote the smooth connection 
$$
\tilde{\nabla}= \nabla_0+ t. (\nabla_1-\nabla_0)
$$
 on $q^*E$, $t\in I$. 
Denote the curvature form of $\tilde{\nabla}$ by $\tilde{\Theta}$. Integration of the Chern form $P_p(\tilde{\Theta},...,\tilde{\Theta})$, along the fibres of $q$ gives the form:
$$
TP(\tilde{\nabla}):= \int_I P_p(\tilde{\Theta},...,\tilde{\Theta}) 
$$
 of degree $2p-1$ on $X$.

Furthermore, since $P_p(\tilde{\Theta},...,\tilde{\Theta})$ is a closed form, and applying Stokes theorem, we observe that
$$
dTP(\tilde{\nabla})= P_p(\nabla_1^2,...,\nabla_1^2) - P_p(\nabla_0^2,...,\nabla_0^2).
$$

%is closed, by Stokes theorem and using the fact that $P_p(\nabla_i^2,...,\nabla_i^2)=0$ for $i=0,1$.

Suppose $\gamma:=\{\nabla_t\}_{t\in I}$ is a path of flat connections  on $E$ joining $\nabla_0$ and $\nabla_1$.
Consider the $2p-1$ form obtained by Cheeger-Simons \cite[Proposition 2.9]{Ch-Sm}:
$$
\eta_p:= p.\int_{I} P_p(\f{d}{dt}\nabla_t, \nabla_t^2,...,\nabla_t^2) \,\,dt,
$$ 
which satisfies $d\eta_p= P_p(\nabla_1^2,...,\nabla_1^2) - P_p(\nabla_0^2,...,\nabla_0^2)$.
In particular, $\eta_p$ is a closed form whenever $\nabla_0$ and $\nabla_1$ are flat.

Then we have the relation
\begin{equation}\label{betaform}
TP(\tilde{\nabla})- \eta_p = d\beta
\end{equation}
for an exact form $d\beta$. This is because, we can carry out this computation on a large Grassmannian (as in Theorem \ref{canDC}), for which the odd degree cohomologies are zero. See also \cite[Proposition 3.6, p.53]{Chn-Sm}.

The map given by integration along the fibres
\begin{equation}\label{homotopyforms}
B:A^{2p}(I\times X)_{cl} \rar A^{2p-1}(X),
\end{equation}
evaluated on the $p$-th Chern form is now $B(P_p(\tilde{\Theta},...,\tilde{\Theta}))=\eta_p + d\beta$.

Now we want to apply 'integration along the fibres' for differential characters of $(q^*E,\tilde{\nabla})$.
In general, there is a map \cite[Proposition 2.1]{Freed}:
$$
\hat{B}: \widehat{H^{2p}}(T\times X) \rar \widehat{H^{2p-k}}(X)
$$
if $T$ is a $k$-dimensional smooth manifold without a boundary, and it is compatible with the map $B$ in \eqref{homotopyforms}, on closed forms.

Even if $T$ has boundary, and for a $2p$-differential character $\hat{\omega}$ on $T\times X$ lifting a $2p$-differential form $\omega$, if the form $B(\omega)$ is closed then the same proof of \cite[Proposition 2.1]{Freed} gives a well-defined element 
$$
\hat{B}(\hat{\omega}) \in \widehat{H^{2p-k}}(X).
$$

%(This is independent of the path chosen between $\nabla_0$ and $\nabla_1$, upto a closed form.)

In our situation, since $T=I$ and $\eta_p$ is closed, the image of the differential character $\hat{c_p}(q^*E,\tilde{\nabla})$, 
$$
\hat{B}(\hat{c_p}(E,\tilde{\nabla}))\in \widehat{H^{2p-1}}(X)
$$ 
is well-defined. 
Now we note that the form $d\beta$ (see \eqref{betaform}) has a unique lifting 
$$
\hat{d\beta}\in \widehat{H^{2p-1}}(X)
$$
given by the linear functional 
$$
\int\beta:Z_{2p-2}\rar \C/\Z.
$$
Hence the difference of the differential characters
$$
\hat{B}(\hat{c_p}(E,\tilde{\nabla}))\,-\,\hat{d\beta}\, \in \widehat{H^{2p-1}}(X)
$$
lifts the form $\eta_p$, canonically.

Since we assume that $\{\nabla_t\}_{t\in I}$ is a family of flat connections, we note that the form $\eta_p$  is identically zero, whenever $p\geq 2$. 

Hence the difference  
 $$
 \hat{B}(\hat{c_p}(E,\tilde{\nabla}))- \hat{d\beta}\,\in\, \widehat{H^{2p-1}}(X)
$$
has zero projection in $A^{2p-1}(X)$.
Hence this differential character in fact lies in the cohomology group
$$
 \hat{B}(\hat{c_p}(E,\tilde{\nabla}))-\hat{d\beta}\,\in\, H^{2p-2}(X, \C/\Z).
$$
Denote this class by $CS_p(E,\gamma)$.

\end{proof}

In the above proof, we note that the differential character $\hat{B}(\hat{c_p}(E,\tilde{\nabla}))$ lifts the $(2p-1)$-form $TP(\tilde{\nabla})$ and the differential character 
$\hat{B}(\hat{c_p}(E,\tilde{\nabla}))- \hat{d\beta}$ lifts the form $\eta_p$.

\begin{definition}
For any smooth path of connection $\gamma$ between the flat connections $\nabla_0$ and $\nabla_1$, we denote the differential character 
$$
\hat{TP}_p(\tilde{\nabla}):= \hat{B}(\hat{c_p}(E,\tilde{\nabla})),
$$
and
$$
\hat{c_p}(E,\gamma):= \hat{B}(\hat{c_p}(E,\tilde{\nabla}))- \hat{d\beta},
$$
for a uniquely determined form $\beta \in A^{2p-2}(X)/dA^{2p-3}(X)$, as in \eqref{betaform}.

Furthermore, if the path $\gamma$ parametrises flat connections, we call the cohomology classes 
$$
CS_p(E,\gamma) \in H^{2p-2}(X,\C/\Z), \m{ for } p\geq 2
$$ 
the \textit{tertiary classes}, associated to the path $\gamma$.
\end{definition}

%%%%%%%%%%%%%%%%%%%%%%%%%%%%%%%%%%%%%%%%%%%%%%%%%%%%%%%%%%%%%%%%%%%%%%%%%%%%%%%%%%%%%%%%%%%%%%%%%%%%%%%%%%%%%%%%%%%%%%%%%%%%%%%%

\section{Rigidity of the tertiary invariants $CS_p(E,\gamma)$}
%%%%%%%%%%%%%%%%%%%%%%%%%%%%%%%%%%%%%%%%%%%%%%%%%%%%%%%%%%%%%%%%%%%%%%%%%%%%%%%%%%%%%%%%%%%%%%%%%%%%%%%%%%%%%%%%%%%%%%%%%%%%%%%%

Suppose $\nabla_0$ and $\nabla_1$ are two flat connections on a complex vector bundle $E$ on a smooth manifold $X$.
Let $\{\gamma_t\}_{t\in I}$ be a one parameter family of smooth paths of flat connections joining $\nabla_0$ and $\nabla_1$, i.e.,  for each $t$, $\gamma_t$ is a 
smooth path of flat connections joining $\nabla_0$ and $\nabla_1$.

We prove the rigidity property of the corresponding tertiary invariants, i.e.,
$$
CS_p(E,\gamma_{t_0})\,=\, CS_p(E,\gamma_{t_1})
$$
for the end-points $t_0, t_1$ of the path $I$, for appropriate values of $p$.

In other words, the tertiary invariants depend only on the endpoints $\nabla_0$ and $\nabla_1$ and the homotopy class of paths of flat connections between them.

We first give a variational formula.

\begin{proposition}
 Suppose $\{\gamma_s\}_{s\in I}$ is a one parameter family of smooth paths of connections $\{\nabla_{s,t}\}$, joining  two flat connections $\nabla_0$ and $\nabla_1$ on a complex vector bundle $E$, on a smooth manifold $X$.
Consider the corresponding one parameter family of differential characters $\hat{c_p}(E, \gamma_s)$ in $\widehat{H^{2p-1}}(X)$.
Then we have the variational formula:
$$
\frac{d}{ds} \hat{c_p}(E, \gamma_s)\,=\, \int_I P_p(\f{d}{ds}\f{d}{dt} \nabla_{t,s} \wedge (\nabla_{s,t}^2)^{p-1} + (p-1)\f{d}{dt} \nabla_{s,t}\wedge\f{d}{ds}(\nabla_{s,t}^2) \wedge (\nabla_{s,t}^2)^{p-2}).
$$

In particular, 
 if $\{\gamma_s\}$ is a one parameter  family of paths parametrising flat connections, then  the classes:
$$
CS_p(E,\gamma_s) \in H^{2p-2}(X,\C/\Z)
$$
are constant, whenever $p\geq 3$. In other words the tertiary invariants depend only on the homotopy class of a path $\gamma$. 

\end{proposition}
\begin{proof}
We first note the following remark on tangent spaces \cite[Lemma 2.1]{Iy-Si}:
$$
T_0(\widehat{H^{2p-1}}(X))\,=\, \f{A^{2p-2}(X,\C)}{dA^{2p-3}(X,\C)}
$$
given by the identification $(t\int_z \alpha, t.d\alpha) \mapsto \alpha$.

We want to express the form $\eta_p$ as a locally exact form $d\alpha =\eta_p$ for a form $\alpha \in \f{A^{2p-2}(X,\C)}{dA^{2p-3}(X,\C)}$, so that we can calculate the variation of $\alpha$.  

We can now assume that the path $\gamma_0$ (when $s=0$) is a trivial path based at $\nabla_0$. This means that the form $\eta_p=0$, when $s=0$.

Let $\ov{\nabla}$ be a connection on $I\times I\times X$ defined as follows:
$$
s[(1-t)\nabla_0 +t\nabla_1] + (1-s)d
$$
on the pullback bundle $q^* E$, for the projection $q: I\times I \times X \rar X$.
This means that the $(2p-1)$-form $TP_p(\ov{\nabla})_0=0$.

Let $\ov{\nabla}_s$ denote the restriction of the connection $\ov{\nabla}$, on the slice $\{s\}\times I \times X$.

Now consider the $(2p-2)$-form
$$
TP(\ov{\nabla})= \int_I \int_I P(\ov{\nabla}^2,..., \ov{\nabla}^2).
$$
This is the same as the integral
$$
\int_I TP(\ov{\nabla}_s)= \int_I ((\eta_p)_s + d\beta_s).
$$
By Stoke's theorem, we have
$$
dTP(\ov{\nabla}) = TP((\ov{\nabla})_1)- TP((\ov{\nabla})_0).
$$
When $s=0$, $\gamma_0$ is a trivial path, and we note that $TP(\ov{\nabla}_0)=0$.

In particular, when $s=1$, the form $(\eta_p)_1$ is expressed as
$$
dTP(\ov{\nabla}) -d\beta_1  \,=\,  (\eta_p)_1.
$$

We now compute the variation of the family of differential characters :
\begin{eqnarray*}
\f{d}{ds}\hat{c_p}(E, \gamma_s) & = &  \f{d}{ds} (TP(\ov{\nabla})- \beta_s) \\
                                & = & \f{d}{ds}[ (\int_I (\eta_p)_s +d\beta_s) -\beta_s ] \\
                                & = & \f{d}{ds} [ (\int_I (\eta_p)_s) +\beta_1 -\beta_s].
\end{eqnarray*}

Now we note that
\begin{eqnarray*}
\f{d}{ds} \int_I (\eta_p)_s & = & \f{d}{ds} \int_I P_p(\f{d}{dt} {\nabla}_{t,s},  {\nabla}^2_{t,s},..., {\nabla}^2_{t,s}) dt \\
                            & = &   \int_I \f{d}{ds}  P_p(\f{d}{dt} {\nabla}_{t,s},  {\nabla}^2_{t,s},..., {\nabla}^2_{t,s}) dt \\
                            & = & \int_I P_p(\f{d}{ds}\f{d}{dt} \nabla_{t,s} \wedge (\nabla_{s,t}^2)^{p-1} + (p-1)\f{d}{dt} \nabla_{s,t}\wedge\f{d}{ds}(\nabla_{s,t}^2) \wedge (\nabla_{s,t}^2)^{p-2}). 
\end{eqnarray*}

We now observe that if $\{\nabla_{s,t}\}$ parametrises flat connections and if  $p\geq 2$, then this integral is identically zero.

Hence we get the equality:
$$
\f{d}{ds}\hat{c_p}(E, \gamma_s)\,=\, \f{d}{ds}\beta_s.
$$

To deduce that this quantity is zero, we recall the character diagram involving the differential cohomology and the exact sequence of cohomologies associated to the coefficient sequence
$$
0\rar \Z \rar \C \rar \C/\Z \rar 0.
$$
See \cite{Sm-Su}.

\begin{eqnarray*}
 0\,\,\,\,\,\,\, \,\, \,\,\,\, \,\,\,\,\,     &      &\,\,\,\,\,\,\,\,\,\,\,\,\, \,\,\,0 \\
\searrow \,\,\,\,\,\,\,\,&       & \,\,\,\,\,\,\,\,\nearrow \\
 H^{k-1}(\C/\Z) &\rar & H^k(\Z) \\
 \nearrow\,\,\,\,\, \,\,\,  & \searrow \,\,\,\,\,\,\,\,\,\,\,\nearrow & \,\,\,\,\,\,\,\, \searrow \,\,\,\,\,\,\,\nearrow\\
H^{k-1}(\C) \,\,\,\,& \hat{H^{k}}(X)  & \,\,\,\,\,\,\,\,\,H^k(\C) \\
\nearrow\,\,\,\,\,\,\,\,\,\, \searrow \,\,\,\,\,\,\,\,\,\,\,& \nearrow \,\,\,\,\,\,\,\,\searrow  & \,\,\,\,\nearrow \,\,\,\,\,\,\,\,\,\,\,\,\,\,\searrow\\
  \f{A^{k-1}}{A^{k-1}_\Z} &\sta{d}{\rar} & A^k_\Z \\
 \nearrow\,\,\,\,\,\,\,\,\, & & \,\,\,\,\,\,\,\searrow \\
  0\,\,\,\,\,\,\, \,\, \,\,\,\, \,\,\,\,\,     &      &\,\,\,\,\,\,\,\,\,\,\,\,\, \,\,\,0 .\\
\end{eqnarray*}

In our situation, $k=2p-1$ and the forms $\beta_s$ lie in the group $A^{k-1}/A^{k-1}_\Z$ and under the map $d$, the forms $d\beta_s$ lie in $A^k_\Z$.
Since the tangent space of $A^k_\Z$ is equal to zero, the variation of $\beta_s$ is the same as the variation of $d\beta_s$, which is zero. This implies that the form $d\beta_s$ is constant in a variation over $s$.

Hence we have shown that 
$$
\f{d}{ds}\hat{c_p}(E, \gamma_s)\,=\,0
$$
if $\gamma_s$ is a one parameter family of paths parametrising flat connections, and whenever $p\geq 3$.

\end{proof}

%%%%%%%%%%%%%%%%%%%%%%%%%%%%%%%%%%%%%%%%%%%%%%%%%%%%%%%%%%%%%%%%%%%%%%%%%%%%%%%%%%%%%%%%%%%%%%

\section{Chern-Simons classes in compactly supported cohomology}\label{compactsupp}
%%%%%%%%%%%%%%%%%%%%%%%%%%%%%%%%%%%%%%%%%%%%%%%%%%%%%%%%%%%%%%%%%%%%%%%%%%%%%%%%%%%%%%%%%%%%%%

Suppose $(E,\nabla)$ is a smooth connection on a smooth variety or manifold. 

We observe that differential characters of $(E,\nabla)$ can be defined in the 'compactly supported' differential cohomology as follows. In particular, we have
\begin{proposition}
 Given a flat connection $(E,\nabla)$ on a smooth manifold, the Chern-Simons classes $CS_p(E,\nabla) \in H^{2p-1}(X,\C/Z)$ can be lifted canonically into the compactly supported cohomology
$$
CS_p(E,\nabla) \in H^{2p-1}_c(X,\C/Z).
$$
\end{proposition}
\begin{proof}

Denote
$$
\widehat{H^{2p}_c}(X):= \{ (f,\alpha): f: Z_{2p-1}(X)\rar \C/\Z,\, \delta(f)= \alpha, \alpha \m{ is a closed compactly supp.form}\}.
$$
Here $Z_{2p-1}(X)$ is the group of infinite $(2p-1)$- closed chains.
 
Then there is a short exact sequence:
$$
0\rar H^{2p-1}_c(X, \C/\Z) \rar \widehat{H^{2p}_c}(X) \rar A^{2p}_{ccl,\Z} \rar 0.
$$
Here $A^{2p}_{ccl,\Z}$ denotes the space of compactly supported closed forms of degree $2p$ on $X$ with integral periods.

Suppose $\nabla$ takes values in $A^1_c(X)$ (instead of $A^1(X)$), then the usual theory of differential characters (see Theorem \ref{canDC}), will define classes:
$$
\widehat{c_p}(E,\nabla)\,\in\, \widehat{H^{2p}_c}(X).
$$

This is because in the universal situation $(\cE,\tilde{\nabla})$ lives on a  Grassmannian, for which the compactly supported forms/cohomology is the same as usual forms/cohomology.

Now suppose $\nabla$ is flat, or whenever the Chern form $c_p(E,\nabla)$ is identically zero, then it defines a class:
$$
\widehat{c_p}(E,\nabla)_c\,\in\, H^{2p-1}_c(X,\C/\Z).
$$

It is clear that under the usual map of cohomologies:
$$
H^{2p-1}_c(X,\C/Z) \rar H^{2p-1}(X,\C/Z)
$$
we have
$$
\widehat{c_p}(E,\nabla)_c \mapsto \widehat{c_p}(E,\nabla).
$$
\end{proof}

%%%%%%%%%%%%%%%%%%%%%%%%%%%%%%%%%%%%%%%%%%%%%%%%%%%%%%%%%%%%%%%%%%%%%%%%%%%%%%%%%%%%%%%%%%%%%%%%%%%%%%%%%%%%%%%%%%%
%%%%%%%%%%%%%%%%%%%%%%%%%%%%%%%%%%%%%%%%%%%%%%%%%%%%%%%%%%%%%%%%%%%%%%%%%%%%

%%%%%%%%%%%%%%%%%%%%%%%%%%%%%%%%%%%%%%%%%%%%%%%%%%%%%%%%%%%%%%%%%%%%%%%%%%%%%%%%%%%%%%%%%%%%%%%%%%%%%%%%%%%%%%%%%%%%%%%%%%%%%%

\end{document}